\documentclass[12pt]{amsart}
\usepackage{amsmath}
\usepackage{amsthm}
\usepackage{amssymb}
\usepackage{amscd}
\usepackage{hyperref}
\usepackage{color}

\newcommand{\RR}{\mathbb{R}}

\newcommand{\hght}{\operatorname{ht}}
\newcommand{\Spec}{\operatorname{Spec}}

\newcommand{\length}{\ell}
\newcommand{\pf}{F_*}

\newcommand{\Supp}{\operatorname{Supp}}

\newcommand{\eh}{\operatorname{e}}
\newcommand{\ehk}{\eh_{HK}}
 
\newcommand{\mf}{\mathfrak}
\newcommand{\frq}[1]{{#1}^{[q]}}
\newcommand{\ul}[1]{\underline{#1}}

\renewcommand{\hat}{\widehat}

\newcommand{\gr}{\operatorname{gr}}

\newtheorem{theorem}{Theorem}
\newtheorem{lemma}[theorem]{Lemma}
\newtheorem{proposition}[theorem]{Proposition}
\newtheorem{corollary}[theorem]{Corollary}

\newtheorem{conjecture}[theorem]{Conjecture}
\newtheorem*{statement*}{Statement}
\newtheorem*{theorem*}{Theorem}
\newtheorem*{lemma*}{Lemma}
\newtheorem*{fact*}{Fact}

\theoremstyle{definition}
\newtheorem{definition}[theorem]{Definition}
\newtheorem*{definition*}{Definition}
\newtheorem{example}[theorem]{Example}
\newtheorem*{example*}{Example}

\theoremstyle{remark}
\newtheorem{remark}[theorem]{Remark}
\newtheorem*{remark*}{Remark}

\begin{document}
\title[Upper semi-continuity of the Hilbert-Kunz multiplicity]
{Upper semi-continuity of the Hilbert-Kunz multiplicity}

\author{Ilya Smirnov}
\address{Department of Mathematics\\
University of Virginia\\
 Charlottesville, VA 22904-4137 USA}
\email{is6eu@virginia.edu}

\date{\today}

\begin{abstract}
We prove that the Hilbert-Kunz multiplicity is upper semi-continuous in F-finite rings
and algebras of essentially finite type over an excellent local ring.
\end{abstract}


\maketitle
\tableofcontents

\section{Introduction}\label{intro}
Let $R$ be a commutative Noetherian ring of characteristic $p > 0$.
For a prime ideal $\mf p$, the (normalized) Hilbert-Kunz function of $\mf p$ is defined to be 
\[ e \mapsto \frac{\length_{R_\mf p} (R_\mf p/\mf p^{[p^e]}R_\mf p)}{p^{ed}},\]
where $e$ is a positive integer and $\mf p^{[p^e]}$ is the ideal generated by the $p^e$th powers of elements of $\mf p$.
The Hilbert-Kunz multiplicity of $\mf p$ is defined as the limit of this sequence.
Hilbert-Kunz theory originates in the work of Kunz (\cite{Kunz1, Kunz2}) 
and Hilbert-Kunz multiplicity was introduced by Monsky (\cite{Monsky}) in 1983.  

From the beginning, there was a perception that Hilbert-Kunz theory should be a measure of singularities.
In fact, in 1969 Kunz has shown that Hilbert-Kunz function characterizes singularity,
and then, in 2000, Watanabe and Yoshida generalized this for Hilbert-Kunz multiplicity.
In \cite{WatanabeYoshida} they showed that an unmixed local ring $(R, \mf m)$ is regular if and only if $\ehk(\mf m) = 1$.

Since it follows from the work of Kunz that $\ehk(R) \geq 1$, one may expect that when 
$\ehk(R)$ is getting close to $R$, the singularity of $R$ is getting better.
A notable example of this was given by Blickle and Enescu in \cite{BlickleEnescu}
and then improved by Aberbach and Enescu in \cite{AberbachEnescu}.
They show that if $\ehk(R)$ is sufficiently close to $1$, then $R$ has to be Gorenstein and F-regular.

This work is devoted to a global property of Hilbert-Kunz multiplicity: upper semi-continuity. 
A formal definition of upper semi-continuity is given in Definition~\ref{upsc}, but via Nagata's criterion 
(Proposition~\ref{nagcrit}) this question can be considered as a distribution property of singularities of local rings of $R$.
Namely, given a prime ideal $\mf p$ of $R$, we want to know if the Hilbert-Kunz multiplicity of 
a generic prime containing $\mf p$ is close to the Hilbert-Kunz multiplicity of $R_\mf p$, see Proposition~\ref{closesing}. 
So, we may say, that we are trying to understand if Hilbert-Kunz multiplicity sees that
the singularity of a general prime containing $\mf p$ is close to that of $\mf p$.

We also want to bring reader's attention to Example~\ref{badhk},
which shows that the Hilbert-Kunz multiplicity of a general prime containing $\mf p$
need not to be equal to $\ehk(\mf p)$. This illustrates the subtlety of the problem.
Also, note that the corresponding statement for the Hilbert-Samuel multiplicity is true.

In 1976 (\cite{Kunz2}) Kunz proved that for a fixed $e$ the Hilbert-Kunz function 
$\length(R_\mf p/\mf p^{[p^e]}R_\mf p)/p^{ed}$ is upper semi-continuous, 
but his methods do not extend for the limit. 
In \cite{EnescuShimomoto} Enescu and Shimomoto asked whether Hilbert-Kunz multiplicity is upper semi-continuous.
Also they proved that Hilbert-Kunz multiplicity is dense upper semi-continuous on the maximum spectrum, 
which is a weaker condition. 
Later, in 2011, there was a group working on this question in the AIM workshop "Relating test ideals and multiplier ideals".  

We prove that Hilbert-Kunz multiplicity is upper semi-continuous in locally equidimensional F-finite rings and 
locally equidimensional rings of essentially finite type over an excellent local ring, 
a mild restriction that is satisfied by complete local domains and domains finitely generated over a field.
To achieve this, we want to control the convergence rate of the Hilbert-Kunz function and, 
building on Tucker's estimates from \cite{Tucker}, 
we show that it can be controlled generically in Theorem~\ref{funiconv} for F-finite domains and in Theorem~\ref{essuniconv}
for algebras of essentially finite type over a complete local domain.
This method allows us to reduce upper semi-continuity of the limit of a sequence (Hilbert-Kunz multiplicity)  
to upper semi-continuity of a term of the sequence (a fixed Hilbert-Kunz function), 
and the latter is known by the work of Kunz.
This strategy should be useful for other numerical invariants in positive characteristic, 
for example for F-signature. 
However, an application  of this approach seems to require a better understanding of F-signature, 
see Remark~\ref{F-signature remark}.

The structure of the proof is as following. We start with general preliminaries in Section~\ref{prelim}. 
Then we develop the machinery of global convergence estimates in Sections~\ref{ffinite estimate} and \ref{ess estimate}.
Section~\ref{ffinite estimate} treats F-finite case and Section~\ref{ess estimate} obtains the same result 
for algebras of essentially finite type over a complete local domain, so a reader interested only in F-finite case 
may skip it. Then the estimates are used in Section~\ref{final} to prove the main theorem.

\section{Preliminaries}\label{prelim}
In this paper all rings assumed to be commutative Noetherian and containing an identity element. 
For a module $M$ over a ring $R$, we will use $\length (M)$ to denote the length of $M$. 

Let $R$ be a ring of characteristic $p > 0$. For convenience, we use $q = p^e$ where $e$ may vary.
For an ideal $I$ of $R$, let $\frq{I}$ be the ideal generated by $q$th powers of the elements of $I$.
By $\pf R$ we mean $R$ viewed as an $R$-module via the extension of scalars through the Frobenius endomorphism.
If $R$ is reduced, $\pf R$ can be identified with the ring of $p$-roots $R^{1/p}$.
We say that $R$ is F-finite if $\pf R$ is a finitely generated $R$-module.

\begin{definition}\label{upsc}
Let $X$ be a topological space.
A real-valued function $f$ is upper semi-continuous if for any $a \in \RR$
the set $\{x \in X\mid f(x) < a\}$ is open in $X$.
\end{definition}

In his papers \cite{Kunz1} and \cite{Kunz2}, Kunz initiated the study of the Hilbert-Kunz function $f_q(R) = \frac{1}{q^d} \length(R/\frq{\mf m})$.
For any $q$, we can define a function on $\Spec R$, the spectrum of $R$, by setting 
$f_q(\mf p) = f_q(R_\mf p)$.
In \cite[Proposition~3.3, Corollary~3.4]{Kunz2} Kunz obtained the following results.

\begin{theorem}\label{qupsemi}
If $R$ is a locally equidimensional ring, then for all $q$
\begin{enumerate}
\item $f_q(\mf p) \leq f_q(\mf q)$, if $\mf p \subseteq \mf q$,
\item $f_q(\mf p)$ is upper semi-continuous on $\Spec R$.
\end{enumerate}
\end{theorem}
We should note, that Kunz claimed this result for an equidimensional ring, but
Shepherd-Barron pointed out in \cite{ShepherdBarron} that the theorem is false if $R$ is not locally equidimensional.

\begin{definition}
Let $(R, \mf m)$ be a local ring of characteristic $p > 0$.
It was shown by Monsky that the limit 
\[
\ehk(R) = \lim_{q \to \infty} \frac{\length \left (R/\frq{\mf m} \right)}{q^{\dim R}}
\]
exists and is called the Hilbert-Kunz multiplicity of $R$.
\end{definition}

Defining $\ehk(\mf p) := \ehk(R_\mf p)$, we can view the Hilbert-Kunz multiplicity as a function on $\Spec R$.
In view of Kunz's result, it was natural to pose the following conjecture.

\begin{conjecture}\label{semi}
If $R$ is a locally equidimensional excellent ring,
then the Hilbert-Kunz multiplicity is an upper semi-continuous function on $\Spec R$.

Or, less generally, 
if  $R$ be a locally equidimensional F-finite ring,
then the Hilbert-Kunz multiplicity is an upper semi-continuous function on $\Spec R$.
\end{conjecture}
Note that F-finite rings are excellent by a theorem of Kunz (\cite{Kunz2}[Theorem 2.5]).

Also, we want to note that the restriction of the second statement in the conjecture is 
somehow natural, since F-finite rings are much easier to work with. 
For example, it is still not known whether all excellent rings have a test element.

\begin{remark}
The reader should be warned that Shepherd-Barron (in \cite{ShepherdBarron}[Corollary 2]) 
claimed a much stronger statement. 
However, in his proof, he used that a descending sequence of open sets stabilizes without a proper justification.
In fact, Shepherd-Barron's claim implies that $\ehk$ attains only finitely many values on $\Spec R$. 
But the author was able to give a counter-example to this claim, see Example~\ref{badhk}.
\end{remark}

It is also worth pointing out that a stronger property holds for Hilbert-Kunz functions.
The proofs of semi-continuity by Kunz (\cite{Kunz2}) and Shepherd-Barron (\cite{ShepherdBarron}) show the following statement.

\begin{proposition}\label{locconstant}
Let $R$ be a locally equidimensional ring and $\mf p$ be a prime ideal.
Then for any fixed $q$ and any $a \in \RR$ the set
\[
\{\mf p\mid f_q(\mf p) \leq a\}
\]
is open in $\Spec R$.
\end{proposition}

\begin{example}\label{badhk}
Let $R = F[x, y, z, t]/(z^4 +xyz^2 + (x^3 + y^3)z + tx^2y^2)$, where 
$F$ is the algebraic closure of $\mathbb Z/2\mathbb Z$.
In \cite{BrennerMonsky}, Brenner and Monsky showed that tight closure does not commute with localization in  
this ring. 

Let $\mf p = (x, y, z)$, this is  a prime ideal of dimension one in $R$. 
Building on the work of Monsky (\cite{MonskyQP}), the author was able to show
that the Hilbert-Kunz multiplicity attains infinitely many values on the set of prime ideals containing $\mf p$. 
Moreover, for any prime ideal $\mf m$ containing $\mf p$,
$\ehk(\mf m) > \ehk(\mf p)$, so the set $\{\mf q\mid \ehk(\mf q) \leq \ehk(\mf p)\}$ is not open.
The details will appear in a future paper.

This rather surprising result shows that, compared to the Hilbert-Samuel multiplicity and a fixed Hilbert-Kunz function,
the Hilbert-Kunz multiplicity has a distinctively different global behavior. 
\end{example}

We use the following standard terminology: 
a closed set $V(I)$ consists of all prime ideals containing $I \subseteq R$,
a principal open set $D_s$ consists of all prime ideals not containing $s \in R$. 
Since $D_s$ can be naturally identified with $\Spec R_s$, 
we will sometimes abuse notation, and, by saying to invert an element $s$, 
we will mean to consider $D_s$.

We recall Nagata's criterion of openness in $\Spec R$ (\cite[22.B]{Matsumura}).
\begin{proposition}\label{nagcrit}
Let $R$ be a ring. 
A subset $U$ of $\Spec R$ is open if and only if
\begin{enumerate}
\item $U$ is stable under generalization, i.e. if $\mf q \in U$ and $\mf p \subseteq \mf q$, then $\mf p \in U$, 
\item $U$ contains a nonempty open subset of $V(\mf p)$ for any $\mf p \in U$.
\end{enumerate}
\end{proposition}

Since any open set is a union of principal open sets, the second condition is equivalent to 
$U \supseteq V(\mf p) \cap D_s \neq \emptyset$ for some $s$.

From Theorem~\ref{qupsemi} it follows that if $\mf p \subseteq \mf q$ then
$\ehk(\mf p) \leq \ehk(\mf q)$, so for any $a$ 
the set $\{\mf p \in \Spec R\mid \ehk(\mf p) < a\}$ is stable under generalization.
Hence, it is enough to verify only the second condition of the criterion.
Thus we can restate Conjecture~\ref{semi} in the following form.

\begin{proposition}\label{closesing}
Let $R$ be a locally equidimensional ring.
Then the Hilbert-Kunz multiplicity is upper semi-continuous on $\Spec R$ if and only if 
for any prime ideal $\mf p$ and any $\varepsilon > 0$ 
there exists $s \notin \mf p$ such that for all prime ideals $\mf q \in D_s \cap V(\mf p)$ 
\[\ehk (\mf q) <  \ehk(\mf p) + \varepsilon.\]
\end{proposition}

It is easy to show that we can restrict ourselves to domains.

\begin{proposition}\label{godomain}
Let $R$ be a locally equidimensional ring. If the Hilbert-Kunz multiplicity is upper semi-continuous 
in $R/\mf p$ for all minimal primes $\mf p$ of $R$, then 
the Hilbert-Kunz multiplicity is upper semi-continuous in $R$.
\end{proposition}
\begin{proof}
Given $\varepsilon$, we want to find an element $s \notin \mf p$, 
such that for any ideal $\mf q$ containing $\mf pR_s$ of $R_s$, $\ehk(\mf q) < \ehk(\mf p) + \varepsilon$.

For $i = 1\ldots n$ let $\mf p_i$  be the minimal primes of $R$.
Inverting an element, we may assume that all $\mf p_i$ are contained in $\mf p$.
By the assumption, there exist elements $s_i \notin \mf p$, such that in the corresponding subsets of $\Spec R/\mf p_i$, 
\[
\ehk(\mf qR/\mf p_i) < \ehk(\mf pR/\mf p_i) + \varepsilon/(n\length_{R_{\mf p_i}}\left (R_{\mf p_i})\right ).
\]
Now, if we invert the product $s$ of $s_i$, we obtain that 
for any ideal $\mf q$ of $R_s$ that contains $\mf p$, 
by the associativity formula for Hilbert-Kunz multiplicity,  
\begin{align*}
\ehk(\mf q) &= \sum_{i = 1}^n \ehk(\mf q R/\mf p_i)\length_{R_{\mf p_i}}\left (R_{\mf p_i} \right ) <\\ 
&<  \sum_{i = 1}^n 
\left ( \ehk (\mf pR/\mf p_i) + \frac{\varepsilon}{n\length_{R_{\mf p_i}}\left ( R_{\mf p_i} \right)}\right )
\length_{R_{\mf p_i}}\left (R_{\mf p_i} \right)
= \ehk(\mf p) + \varepsilon.
\end{align*}

\end{proof}

\begin{corollary}
Conjecture~\ref{semi} holds if and only if 
for any excellent domain $R$, prime ideal $\mf p$ of $R$, and $\varepsilon > 0$,
there exists $s \notin \mf p$ such that for all prime ideals 
$\mf q \in V(\mf p) \cap D_s$, 
\[\ehk (\mf q) <  \ehk(\mf p) + \varepsilon.\]
\end{corollary}
\begin{proof}
We just note that a quotient of an excellent ring is excellent.
\end{proof}

A descent of semi-continuity over a faithfully flat extension would be extremelly useful for the proof; 
in fact, it would eliminate the need of Section~\ref{ess estimate}.
Unfortunately, there is no good relation between the Hilbert-Kunz multiplicity of 
a local ring and its arbitrary faithfully flat extension. 
So, we are able to obtain a descent statement only for extensions with regular fibers.
Still, the following lemma will be needed in the proof of our main result 
for algebras of essentially finite type over an excellent local ring. 

\begin{lemma}\label{regdescent}
Let $R$ be a ring and $f\colon R \to S$ be a faithfully flat $R$-algebra.
Moreover, suppose $f$ has regular fibers. 
Then Hilbert-Kunz multiplicity is upper semi-continuous in $S$ if and only if it is upper semi-continuous in $R$. 
\end{lemma}
\begin{proof}
Let $Q$ be any prime in $S$ and let $\mf p = Q \cap R$. 
Note that $R_\mf p \to S_Q$ is faithfully flat with regular fibers, so, by a result of Kunz (\cite[Proposition~3.9]{Kunz2}),
$\ehk(R_\mf p) = \ehk(S_Q)$. Thus, under our assumption, the Hilbert-Kunz multiplicity is constant in fibers.

Suppose upper semi-continuity holds in $S$.
Let $a$ be any real number and consider the closed set 
$V(I) = \{Q \mid Q \in \Spec S,\, \ehk(Q) \geq a\}.$
The argument above tells us that for any $Q \in V(I)$ any minimal prime $(Q \cap R)S$ is also in $V(I)$.
Hence we get that $V(I) = V(JS)$ where $J = I \cap R$.

We claim that $V(J) = \{\mf p \mid \mf p \in \Spec R,\, \ehk(\mf p) \geq a \}$.
Note that $J \in \mf p$ if and only if
$JS \subseteq Q$ for any prime $Q$ in $S$ that contracts to $\mf p$, i.e.  $\ehk(\mf p) = \ehk(Q) \geq a$.

For the other direction, note that $f^*\colon \Spec S \to \Spec R$ is surjective,
so, since $\ehk$ is constant in fibers, we obtain that 
\[
\{Q \mid Q \in \Spec S, \ehk(Q) < a\} = (f^*)^{-1}\{\mf p \mid \mf p \in \Spec R, \ehk(\mf p) < a \}.
\]
Hence it is open. 
 
\end{proof}



\section{Globally uniform Hilbert-Kunz estimates for F-finite rings}\label{ffinite estimate}
In this section we essentially rebuild Tucker's uniform Hilbert-Kunz estimates from \cite{Tucker} 
in order to control the rate of convergence of the Hilbert-Kunz function on an open subset.

We will need the following facts about the Hilbert-Samuel multiplicity. 
See \cite[Proposition~11.1.10, Theorem~11.2.4, Proposition~11.2.9]{HunekeSwanson} for proofs.
\begin{proposition}\label{hs}
Let $(R, \mf m)$ be a local ring of dimension $d$, $\ul x$ be a system of parameters
and $I$ an arbitrary $\mf m$-primary ideal. 
\begin{enumerate}
\item $\length \left (R/ \ul x \right ) \geq \eh(\ul x, R)$. If $\ul x$ is a regular sequence, then equality holds.
\item (Associativity formula) $\eh (I, R) = \sum_\mf p \eh(I, R/\mf p) \length_{R_\mf p}\left (R_\mf p/IR_\mf p \right)$,
where the sum is taken over all primes $\mf p$, such that $\dim R/\mf p = \dim R$.\label{ashs}
\item For any numbers $n_1, \ldots, n_d$, $\eh((x_1^{n_1}, \ldots, x_d^{n_d}), R) = n_1\ldots n_d \eh(\ul{x}, R)$.
\end{enumerate}
\end{proposition}

\begin{lemma}[Key lemma]\label{keylemma}	
Let $R$ be an excellent ring of characteristic $p>0$ and $\mf p$ a prime ideal of $R$. 
Let $M$ be a finite $R$-module. There exists a constant $C$(depending only on $M$) 
and an element $s \notin \mf p$, such that for any prime ideal $\mf q \in D_s \cap V(\mf p)$ and for all $q$,
we have
\[ \length_{R_\mf q} \left (M_\mf q/\frq{\mf q}M_\mf q\right ) \leq Cq^{\dim M_\mf q}.\]
\end{lemma}
\begin{proof}

Assume that $M = R/P$ is a cyclic module for a prime ideal $P$.
If $\mf p$ does not contain $P$, we can invert any $s \in P \setminus \mf p$, so
$M_s = 0$ and the assertion is trivially true. Hence, assume $P \subseteq \mf p$.

First, invert an element to make $R/\mf p$ regular; this is possible since $R/\mf p$ is an excellent domain.
Let $S = R/P$, then $S/\mf pS \cong R/\mf p$ is regular too.

Consider the associated graded ring $\gr_\mf p (S) = \oplus_n \mf p^nS/\mf p^{n+1}S$. 
This is a finitely generated $S/\mf pS$-algebra, so by Generic Freeness (\cite[22.A]{Matsumura}), 
we can invert an element of $S/\mf pS$ and make it free over the regular ring $S/\mf pS$. 
It follows that $\mf p^nS/\mf p^{n+1}S$ are projective $S/\mf pS$-modules for all $n$.
Hence, by induction, using the sequences
\[
0 \to \mf p^nS/\mf p^{n+1}S \to S/\mf p^{n+1}S \to S/\mf p^n \to 0,
\]
we get that all residue rings $S/\mf p^nS$ are Cohen-Macaulay in this localization.

Let $\mf q$ be an arbitrary prime ideal in the obtained localization that contains $\mf p$.
Since $R_\mf q/\mf pR_\mf q$ is a regular local ring, 
there exists a system of parameters $\ul x$ that generates $\mf qR_\mf q$ modulo $\mf pR_\mf q$.
Suppose $\mf p$ can be generated by $t$ elements in $R$.
Since $(\mf p^{tq}, \frq {\ul x})R_\mf q \subseteq \frq{\mf q}R_\mf q$, 
\[\length_{R_\mf q} \left (S_\mf q/\frq{\mf q}S_\mf q \right ) 
\leq \length_{R_\mf q} \big (S_\mf q/\big(\mf p^{tq}, \frq{(\ul x)}\big)S_\mf q \big).\]

Since $S/\mf p^{tq}$ are Cohen-Macaulay, 
\[\length_{R_\mf q} \left(S_\mf q/\big(\mf p^{tq}, \frq{(\ul x)}\big)S_\mf q \right) = 
\eh\big(\frq{(\ul x)}, S_\mf q/\mf p^{tq}S_\mf q\big) 
= q^{\hght \mf q/\mf p}\eh\left(\ul x, S_\mf q/\mf p^{tq}S_\mf q\right).
\]
Thus, using the associativity formula, we get
\[\length_{R_\mf q} \left(S_\mf q/\big(\mf p^{tq}, \frq{(\ul x)}\big)S_\mf q \right ) = 
q^{\hght \mf q/\mf p}\eh\left(\ul x, S_\mf q/\mf pS_\mf q\right) \length_{R_\mf p} \left(S_\mf p/\mf p^{tq}S_\mf p \right)
= q^{\hght \mf q/\mf p} \length_{R_\mf p}\left(S_\mf p/\mf p^{tq}S_\mf p\right).\] 
Note, that $\eh(\ul x, S_\mf q/\mf pS_\mf q) = 1$, 
since $\ul x$ generates the maximal ideal of a regular local ring $S_\mf q/\mf pS_\mf q$. 

To finish the argument, note that 
$\length_{R_\mf p}\left (S_\mf p/\mf p^{n}S_\mf p\right) = \length_{R_\mf p}\left (R_\mf p/(P + \mf p^{n})R_\mf p \right)$ 
is a polynomial in $n$ of degree ${\hght \mf p/P}$
for all sufficiently large $n$, so, clearly, there exists a constant $D$, such that for all $q$
\[
\length_{R_\mf q}\left (S_\mf p/\mf p^{tq}S_\mf p\right ) = \length_{R_\mf q}\left(R_\mf p/(P + \mf p^{tq})R_\mf p\right)
\leq D(tq)^{\hght \mf p/P}.
\]

Thus, we obtained a bound 
\[
\length_{R_\mf q} \left(S_\mf q/\frq{\mf q}S_\mf q \right) \leq 
\length \big(S_\mf q/(\mf p^{tq}, \frq{(\ul{x})})S_\mf q\big) 
\leq q^{\hght \mf q/\mf p} D(tq)^{\hght \mf p/P}
= \big(Dt^{\hght \mf p/P}\big)q^{\hght \mf q/P} = Cq^{\dim S_\mf q}.
\]
Hence, the statement has been proved for $C = Dt^{\hght \mf p/P}$, a constant independent of $\mf q$.

By choosing a prime filtration of $M$ over $R$, we can reduce the general case to $M = R/P$. 
Namely, if $P_i$ are prime ideals appearing in the prime filtration, 
then 
\[\length_{R_\mf q} \left(M_{\mf q}/\frq{\mf q}M_{\mf q}\right) 
\leq \sum_i \length_{R_\mf q} \left((R/{P_i})_{\mf q}/\frq{\mf q}(R/{P_i})_{\mf q}\right).
\]
Since there are finitely many primes $P_i$, 
we can invert finitely many elements in order to force the claim for all $R/P_i$.
Also, note that $\dim M_\mf q$ is the maximum of $\dim R_\mf q/P_iR_\mf q$ over the primes in a prime filtration.
So, 
\[\length_{R_\mf q} \left(M_{\mf q}/\frq{\mf q}M_{\mf q}\right) 
\leq \sum_i \length_{R_\mf q} \left(R_\mf q/({P_i} + \frq{\mf q})R_{\mf q}\right)
\leq \sum_i C_i q^{\hght \mf q/P_i} \leq \left(\sum_i C_i\right) q^{\dim M_\mf q}.
\]
\end{proof}

Using a standard argument (\cite[Lemma 3.3]{Tucker} or \cite[Lemma 1.3]{Monsky}), 
we derive from the Key lemma the following result.

\begin{corollary}\label{minprimes}
Let $R$ be an excellent ring of characteristic $p>0$ and $\mf p$ be a prime ideal of $R$. 
Suppose $M$ and $N$ are finite $R$-modules such that their localizations at every minimal prime are isomorphic.
Then there exists a constant $C$ and an element $s \notin \mf p$, 
such that for any prime ideal $\mf q \in D_s \cap V(\mf p)$ and for all $q$,
we have
\[
|\length_{R_\mf q}\left(M_\mf q/\frq{\mf q}M_\mf q\right) - 
\length_{R_\mf q}\left(N_\mf q/\frq{\mf q}N_\mf q\right)| \leq Cq^{\hght \mf q - 1}.
\]
\end{corollary}
\begin{proof}
By the assumptions, we have an exact sequence
\[
N \to M \to K \to 0,
\]
where $K_P = 0$ for every minimal prime $P$.
By Lemma~\ref{keylemma}, we can find an element $s_1$
such that for some constant $C_1$ and all $\mf q \in D_{s_1} \cap V(\mf p)$
\[
\length_{R_\mf q}\left(M_\mf q/\frq{\mf q}M_\mf q\right) - 
\length_{R_\mf q}\left(N_\mf q/\frq{\mf q}N_\mf q\right) \leq 
\length_{R_\mf q}\left(K_\mf q/\frq{\mf q}K_\mf q\right) \leq C_1q^{\dim K_\mf q}.
\]
Since $K_P = 0$ for any minimal prime $P$, $\dim K_\mf q \leq \hght \mf q - 1$.

To finish the proof, we switch $M$ and $N$ in the first part of the argument,
i.e. apply it to the sequence
\[
M \to N \to L \to 0.
\]
Hence, by inverting an element $s_2$, we will get 
\[
\length_{R_\mf q}\left(N_\mf q/\frq{\mf q}M_\mf q\right) - 
\length_{R_\mf q}\left(M_\mf q/\frq{\mf q}M_\mf q\right) \leq 
\length_{R_\mf q}\left(L_\mf q/\frq{\mf q}L_\mf q\right) \leq C_2q^{\dim L_\mf q}
\leq C_2q^{\hght \mf q - 1},
\]
and the claim follows for $C = \max (C_1, C_2)$ and $s = s_1s_2$.
\end{proof}

\begin{definition}
Let $R$ be a ring of characteristic $p > 0$. 
For a prime ideal $\mf p$ of $R$, we denote 
$\alpha(\mf p) = \log_p [k(\mf p):k(\mf p)^p]$,
where $k(\mf p) = R_\mf p/\mf pR_\mf p$ is the residue field of $\mf p$.
\end{definition}

We will need the following result of Kunz (\cite[2.3]{Kunz2}).
\begin{proposition}\label{kunzdegree}
Let $R$ be F-finite and let $\mf p \subseteq \mf q$ be prime ideals.
Then $\alpha(\mf p) = \alpha(\mf q) + \hght \mf q/\mf p$.
\end{proposition}

\begin{theorem}\label{funiconv}
Let $R$ be an F-finite domain and let $\mf p$ be an arbitrary prime ideal.
Then there exists an element $s \notin \mf p$ such that for any $\varepsilon > 0$
there is $q_0$ such that for all $q > q_0$ 
\[ \left|\length_{R_\mf q} \left(R_\mf q/\frq{\mf q}R_\mf q\right)/q^{\hght \mf q} - \ehk(\mf q)\right| < \varepsilon\]
for all prime ideals $\mf q \in D_s \cap V(\mf p)$. 
\end{theorem}
\begin{proof}
Since $R$ is F-finite, 
$R^{p^{\alpha(0)}}$ and $R^{1/p}$ are isomorphic localized at the minimal prime $0$. 
So, by Corollary~\ref{minprimes},
we can invert an element and obtain a global bound
\[
\left|\length_{R_\mf q}\left(R_\mf q^{p^{\alpha(0)}}/\frq{\mf q}R_\mf q^{p^{\alpha(0)}}\right) - 
\length_{R_\mf q}\left(R^{1/p}_\mf q/\frq{\mf q}R_\mf q^{1/p}\right)\right| 
< Cq^{\hght \mf q - 1},   
\]
for an arbitrary prime ideal $\mf q$ containing $\mf p$. 

Now, to finish the proof, we follow Tucker's argument from \cite{Tucker}. 
Proposition~\ref{kunzdegree} applied to the formula above gives
\[
\left|p^{\hght \mf q + \alpha(\mf q)}\length_{R_\mf q}\left(R_\mf q/\frq{\mf q}R_\mf q\right) - 
p^{\alpha(\mf q)}\length_{R_\mf q}\left(R_\mf q/\mf q^{[qp]}R_\mf q\right)\right| < 
Cq^{\hght \mf q - 1}, \text{so}   
\]
\begin{equation}\label{1formula}
\left|p^{\hght \mf q}\length_{R_\mf q}\left(R_\mf q/\frq{\mf q}R_\mf q\right) - 
\length_{R_\mf q}\left(R_\mf q/\mf q^{[qp]}R_\mf q\right)\right|
< p^{-\alpha(\mf q)}Cq^{\hght \mf q - 1} \leq Cq^{\hght \mf q - 1}.
\end{equation}

Now, we prove by induction on $q'$ that
\begin{equation}\label{indformula}
\left|(q')^{\hght \mf q}\length_{R_\mf q}\left(R_\mf q/\mf q^{[q]}R_\mf q\right) - 
\length_{R_\mf q}\big(R_\mf q/\mf q^{[qq']}R_\mf q\big)\right|<
C(qq'/p)^{\hght \mf q -1}\frac{q' - 1}{p - 1}. 
\end{equation}
The induction base of $q' = p$ is (\ref{1formula}). 
Now, assume that the claim holds for $q'$ and we want to prove it for $q'p$.

First, (\ref{1formula}) applied to $qq'$ gives 
\begin{equation}\label{2formula}
\left|p^{\hght \mf q}\length_{R_\mf q}\left(R_\mf q/\mf q^{[qq']}R_\mf q\right) - 
\length_{R_\mf q}\left(R_\mf q/\mf q^{[qq'p]}R_\mf q\right)\right|
< C(qq')^{\hght \mf q - 1},
\end{equation}
and, multiplying the induction hypothesis by $p^{\hght \mf q}$, we get
\begin{equation}\label{3formula}
\left|(q'p)^{\hght \mf q}\length_{R_\mf q}\Big(R_\mf q/\mf q^{[q]}R_\mf q\Big) - 
p^{\hght \mf q}\length_{R_\mf q}\left(R_\mf q/\mf q^{[qq']}R_\mf q\right)\right|
< C(qq')^{\hght \mf q -1}\frac{pq' - p}{p - 1}.
\end{equation}
Combining (\ref{2formula}) and (\ref{3formula}) results in
\[
\left|(q')^{\hght \mf q}\length_{R_\mf q}\Big(R_\mf q/\mf q^{[q]}R_\mf q\Big) - 
\length_{R_\mf q}\left(R_\mf q/\mf q^{[qq']}R_\mf q\right)\right| <
C(qq')^{\hght \mf q -1} \left(\frac{q'p - p}{p - 1} + 1\right),
\]
and the induction step follows.

Now, dividing (\ref{indformula}) by $q'^{\hght \mf q}$, we obtain
\[
\left|\length_{R_\mf q}\Big(R_\mf q/\mf q^{[q]}R_\mf q\Big) - 
\frac{1}{q'^{\hght \mf q}}\length_{R_\mf q}\left(R_\mf q/\mf q^{[qq']}R_\mf q\right)\right|<
Cq^{\hght \mf q -1}\cdot\frac{q' - 1}{p -1}\cdot\frac{1}{q'p^{\hght \mf q - 1}}
\leq Cq^{\hght \mf q - 1}.
\]

Thus, if we let $q' \to \infty$, we get that 
\[\left|\length_{R_\mf q}\left(R_\mf q/\mf q^{[q]}R_\mf q\right) - q^{\hght \mf q}\ehk(\mf q)\right| < Cq^{\hght \mf q - 1},\]
and the claim follows.

\end{proof}

\section{Uniform estimates for a flat extension}\label{ess estimate}
In this section we prove convergence estimates of Theorem~\ref{funiconv} 
for algebras of essentially finite type over a complete domain. 
To do so, we use existence of a faithfully flat F-finite extension, 
and we relativize the estimates of the previous section to use in the extension. 

\begin{lemma}\label{flat estimate}
Let $R$ be a locally equidimensional excellent ring and $S$ be an $R$-algebra.
Let $I$ be an ideal in $R$, let $M$ be an $S$-module such that $\Supp M \subseteq V(I)$, 
and $\mf p$ be a prime ideal of $R$.
Then there exists an element $s \notin \mf p$ 
and a constant $C$
such that for any prime ideal $\mf q \in V(\mf p) \cap D(s)$ 
and for any prime ideal $Q$ in $S$ minimal over $\mf qS$
\[
\length_{S_Q}\left(M_Q/\frq{\mf q}M_Q\right) \leq Cq^{\hght \mf q - \hght I}\length_{S_Q} \left(S_Q/\mf qS_Q\right).
\]
\end{lemma}
\begin{proof}
If $I$ is not contained in $\mf p$ we can invert an element and make $M$ to be zero.
So assume $I \subseteq \mf p$.

Since $R$ is excellent, we can invert an element $s \notin \mf p$ to make 
$R/\mf p$ regular and $R/(\mf p^n + I)$ to be Cohen-Macaulay for all $n$ (see the proof of Lemma~\ref{keylemma}).
We claim that the required bound holds for this $s$.

By taking a prime filtration of $M$ we reduce the statement to $M = S/J$,
where $J$ is a prime ideal in $S$ that contains $IS$.
So
\[
\length_{S_Q}\left(S_Q/(\frq{\mf q}S + J)S_Q\right) \leq 
\length_{S_Q}\left(S_Q/(\frq{\mf q} + I)S_Q \right).
\]

By tensoring a prime filtration of $R_\mf q/(\frq{\mf q} + I)R_\mf q$ with $S_Q$, we have 
\[
\length_{S_Q}(S_Q/\left(\frq{\mf q} + I)S_Q\right) \leq 
\length_{R_\mf q}\left(R_\mf q/(\frq{\mf q} + I)R_\mf q\right) \length_{S_Q}\left(S_Q/\mf qS_Q\right).
\]

Since $R/\mf p$ is regular, we can write $\mf qR_\mf q = (\mf p + (\ul{x}))R_\mf q$,
where $\ul{x}$ are minimal generators of $\mf q/\mf p$. 
Suppose $\mf p$ can be generated by $t$ elements in $R$, hence
$\mf p^{tq} \subseteq \frq{\mf p}$. Thus
\[\length_{R_\mf q}\Big(R_\mf q/\left(\frq{\mf q} + I\right)R_\mf q\Big) = 
\length_{R_\mf q}\left(R_\mf q/\big(\frq{\mf p}+ \frq{(\ul{x})} + I\big)R_\mf q\right)
\leq 
\length_{R_\mf q}\left(R_\mf q/\big(\mf p^{tq} + \frq{(\ul{x})} + I\big)R_\mf q\right).
\]
Now, since $R/\mf q^{tq} + I$ are Cohen-Macaulay,
\[
\length_{R_\mf q}\left(R_\mf q/\left(\mf p^{tq} + I + \frq{(\ul{x})}\right)R_\mf q\right) =
\eh \big(\frq{(\ul{x})}, R_\mf q/(\mf p^{tq} + I)R_\mf q\big).
\]
Moreover, by the associativity formula,
\[
\eh \big(\frq{(\ul{x})}, R_\mf q/(\mf p^{tq} + I)R_\mf q\big)
= \eh \big(\frq{(\ul{x})}, R_\mf q/\mf pR_\mf q\big) \length_{R_\mf p} \left(R_\mf p/(\mf p^{tq} + I)R_\mf p\right),
\]
and, using that $R/\mf p$ is regular,
\[
\eh \big(\frq{(\ul{x})}, R_\mf q/\mf pR_\mf q\big) \length_{R_\mf p} \left(R_\mf p/(\mf p^{tq} + I)R_\mf p\right)
=  q^{\hght \mf q/\mf p} \length_{R_\mf p} \left(R_\mf p/(\mf p^{tq} + I)R_\mf p\right).
\]
The Hilbert-Samuel polynomial of $R_\mf p/IR_\mf p$ has degree 
$\dim R_\mf p/I = \hght \mf p - \hght I$. Hence we can find a constant $D$ such that 
\[
\length_{R_\mf p} \left(R_\mf p/(\mf p^{tq} + I)R_\mf p\right) \leq D(tq)^{\hght \mf p - \hght I} = Cq^{\hght \mf p - \hght I}
\] 
and the claim follows. 
\end{proof}

We will need the following lemma about the Gamma construction.
It is a step in the proof of \cite[Lemma 6.13]{HochsterHuneke}, 
and a more detailed exposition can be found in Hochster's notes 
(\cite[Theorem, page 139]{Hochster}).

\begin{lemma}\label{gamma}
Let $B$ be a complete local domain and $S$ be a $B$-algebra of essentially finite type.
Suppose $S$ is a domain then there exists a purely inseparable faithfully flat F-finite $B$-algebra $B^\Gamma$
such that $S \otimes_B B^\Gamma$ is a domain.
\end{lemma}

\begin{theorem}\label{essuniconv}
Let $B$ be a complete local domain and $R$ be a domain and a $B$-algebra of essentially finite type.
Let $\mf p$ be an arbitrary prime ideal in $R$, 
then there exists an element $s \notin \mf p$ such that for any $\varepsilon > 0$
there is $q_0$ such that for all $q > q_0$ 
\[ \left|\length_{R_\mf q} \left(R_\mf q/\frq{\mf q}R_\mf q\right)/q^{\hght \mf q} - \ehk(\mf q)\right| < \varepsilon\]
for all prime ideals $\mf q \in D_s \cap V(\mf p)$.
\end{theorem}
\begin{proof}
We apply Lemma~\ref{gamma} to the quotient field $L$ of $R$ and obtain a $B$-algebra $B^\Gamma$.
Note that $S = R \otimes_B B^\Gamma$ is F-finite, so $S^{1/p}$ is a finitely generated $S$-module.

By Lemma~\ref{gamma}, $S \otimes_R L \cong B^\Gamma \otimes_B R \otimes_R L \cong B^\Gamma \otimes_B L$ is a domain.
Since $B^\Gamma$ is purely inseparable over $B$, $B^\Gamma \otimes_B L$ is integral over a field $L$, so it is a field.
Since taking $p$-roots commutes with localization, 
$(S)^{1/p} \otimes_R L \cong (S \otimes_R L)^{1/p}$, 
so it is a free module over the field $S \otimes_R L \cong B^\Gamma \otimes_B L$.
Hence, we can invert an element $f$ of $R$ to make $S_f^{1/p}$ a free module over $S_f$.
Since $R$ is a subring of $S$ and $S \otimes_R L$ is a field, 
$S \otimes_R L$ is the quotient field of $S$, thus, by definition, 
the rank of the free module $(S \otimes_R L)^{1/p}$ is $p^{\alpha(0)}$.

Therefore there exist maps 
\[
0 \to S^{1/p} \to S^{p^{\alpha(0)}} \to M \to 0
\]
and 
\[
0 \to S^{p^{\alpha(0)}} \to S^{1/p} \to N \to 0
\]
such that $\Supp M, \Supp N \subseteq V(fS)$.

Thus, using Lemma~\ref{flat estimate} to $M$ and $N$, we can invert an element $s$ and obtain that, for any prime $\mf q$ containing $\mf p$ and for any minimal prime $Q$ of $\mf qS$,
\[
\length_{S_Q}\left(S_Q^{p^{\alpha(0)}}/\frq{\mf q}S_Q^{p^{\alpha(0)}}\right) - 
\length_{S_Q}\left(S_Q^{1/p}/\frq{\mf q}S_Q^{1/p}\right)
\leq \length_{S_Q}\left(M_Q/\frq{\mf q}M_Q\right) 
\leq C_1q^{\hght \mf q - \hght(f)}\length_{S_Q} \left(S_Q/\mf qS_Q\right),
\]
\[
\length_{S_Q}\left(S_Q^{1/p}/\frq{\mf q}S_Q^{1/p}\right ) - 
\length_{S_Q} \left(S_Q^{p^{\alpha(0)}}/\frq{\mf q}S_Q^{p^{\alpha(0)}} \right)
\leq \length_{S_Q}\left(N_Q/\frq{\mf q}N_Q\right) \leq 
C_2q^{\hght \mf q - \hght(f)}\length_{S_Q} \left(S_Q/\mf qS_Q\right).
\]
Thus, by taking $C = \max(C_1, C_2)$ and noting that $\hght(f) = 1$, we obtain
\[
\left|\length_{S_Q}\left(S_Q^{p^{\alpha(0)}}/\frq{\mf q}S_Q^{p^{\alpha(0)}}\right) - 
\length_{S_Q}\left(S_Q^{1/p}/\frq{\mf q}S_Q^{1/p}\right)\right| 
< Cq^{\hght \mf q - 1}\length_{S_Q} \left(S_Q/\mf qS_Q\right).   
\]
So, since $\alpha (0) = \hght Q + \alpha(Q)$ by Proposition~\ref{kunzdegree},
\[
\left|p^{\hght Q + \alpha(Q)}\length_{S_Q}\left(S_Q/\frq{\mf q}S_Q\right) - 
p^{\alpha(Q)}\length_{S_Q}\left(S_Q/\mf q^{[qp]}S_Q\right)\right| < 
Cq^{\hght \mf q - 1}\length_{S_Q} \left(S_Q/\mf qS_Q\right).   
\]

Note that $S_Q$ is flat over $R_\mf q$ and $\mf qS_Q$ is $Q$-primary.
Hence for any artinian $R_\mf q$-module $M$, 
\[ 
 \length_{S_Q} \left(M \otimes_{R_{\mf q}} S_Q\right) = \length_{R_\mf q} (M) \length_{S_Q} \left(S_Q/\mf qS_Q\right).
\]
Therefore, the estimate above can be rewritten as 
\[
\left|p^{\hght Q + \alpha(Q)}\length_{R_\mf q}\left(R_\mf q/\frq{\mf q}R_\mf q\right) - 
p^{\alpha(Q)}\length_{R_\mf q}\left(R_\mf q/\mf q^{[qp]}R_\mf q\right)\right| < 
Cq^{\hght \mf q - 1}.   
\]
Since $S$ is flat $\hght Q = \hght \mf q$, so we obtain Equation~\ref{1formula} from Theorem~\ref{funiconv}:
\[
\left|p^{\hght \mf q}\length_{R_\mf q}\left(R_\mf q/\frq{\mf q}R_\mf q\right) - 
\length_{R_\mf q}\left(R_\mf q/\mf q^{[qp]}R_\mf q\right)\right| < 
Cp^{-\alpha(Q)}q^{\hght \mf q - 1} \leq Cq^{\hght \mf q - 1};   
\]
and the proof follows the argument in Theorem~\ref{funiconv}.

\end{proof}

\section{Proof of the main result and concluding remarks}\label{final}

Now, we want to finish the proof of upper semi-continuity of 
the Hilbert-Kunz multiplicity for F-finite rings and algebras of essentially finite type over an excellent local ring.
To do this, we verify the second statement of Proposition~\ref{closesing}.

\begin{theorem}\label{mainthm}
Let $R$ be a locally equidimensional ring.
Suppose that $R$ is either F-finite  or is an algebra of essentially finite type over an excellent local ring $B$. 
If $\mf p$ be a prime ideal of $R$, 
then for any $\varepsilon > 0$, there exists $s \notin \mf p$, such that 
for all prime ideals $\mf q \in D_s \cap V(\mf p)$
\[\ehk(\mf q) < \ehk(\mf p) + \varepsilon.\]
\end{theorem}
\begin{proof}
If $R$ is not $F$-finite, first, consider extension $R \to R \otimes_B \hat{B}$.
Since $B$ is excellent, the natural map $B \to \hat{B}$ is regular. 
So, by \cite[Lemma 4, p. 253]{Matsumura}, $R \to R \otimes_B \hat{B}$ 
satisfies the conditions of Lemma~\ref{regdescent}. 
Hence, by Proposition~\ref{closesing} and Lemma~\ref{regdescent}, we assume that $B$ is complete.

Note that the classes of rings that we consider are stable under taking quotients.
So, by Proposition~\ref{godomain}, we can assume that $R$ is a domain.

By Theorem~\ref{funiconv} and Theorem~\ref{essuniconv}, 
there exists an element $s \notin p$ and a fixed power $q_0 = p^e$,
such that for all $\mf q \in D_s \cap V(\mf p)$
\[ \left|\length_{R_\mf q} \left(R_\mf q/\mf q^{[q_0]}R_\mf q\right)/q_0^{\hght \mf q} - \ehk(\mf q)\right| < \varepsilon/2.\]
In particular,
\[ \left|\length_{R_\mf p} \left(R_\mf p/\mf p^{[q_0]}R_\mf p\right)/q_0^{\hght \mf p} - \ehk(\mf p)\right| < \varepsilon/2.\]

Now, we can use Proposition~\ref{locconstant}, 
and obtain a non-empty subset $\mf p \in U \subseteq V(\mf p)$ open in $V(\mf p)$
such that for any $\mf q \in U$, 
\[
\length_{R_\mf q} \left(R_\mf q/\mf q^{[q_0]}R_\mf q\right)/q_0^{\hght \mf q} = f_{q_0} (\mf q) = 
f_{q_0} (\mf p) = \length_{R_\mf p} \left(R_\mf p/\mf p^{[q_0]}R_\mf p\right)/q_0^{\hght \mf p}.
\]
For completeness we are giving a construction of such $U$ below.

Since $R/\mf p^{[q_0]}$ is excellent, its Cohen-Macaulay locus is open (\cite[7.8.3(iv)]{EGA}).
Thus we can find an open subset $U \subseteq V(\mf p)$ containing $\mf p$
such that for any $\mf q \in U$, $(R/\mf p^{[q_0]})_\mf q$ is Cohen-Macaulay 
and $(R/\mf p)_\mf q$ is regular.

Let $\mf q$ be an arbitrary prime in $U$. 
Since $R_\mf q/\mf pR_\mf q$ is regular, $\mf qR_\mf q$ is generated by a regular sequence $\ul{x}$ modulo $\mf pR_\mf q$.
Then, by the associativity formula,
\[ \length_{R_\mf q} \left(R_\mf q/\mf q^{[q_0]}R_\mf q\right)
 = \length_{R_\mf q} \left(R_\mf q/(\mf p^{[q_0]}, (\ul{x})^{[q_0]})R_\mf q\right)
 = \eh\big((\ul{x})^{[q_0]}, R_\mf q/\mf p^{[q_0]}\big) 
= q_0^{\hght \mf q/\mf p}\length_{R_\mf p} \left(R_\mf p/\mf p^{[q_0]}R_\mf p\right).
\]

Thus, we obtain that on $U \cap D_s$, 
$\length_{R_\mf p} \left(R_\mf p/\mf p^{[q_0]}R_\mf p\right)/q_0^{\hght \mf p}$
is within $\varepsilon/2$ from both $\ehk(\mf p)$ and $\ehk(\mf q)$ and the statement follows. 
\end{proof}

\begin{corollary}
Let $R$ be a locally equidimensional ring.
Moreover, suppose that either $R$ is F-finite  or is an algebra of essentially finite type over an excellent local ring $B$.
Then the Hilbert-Kunz multiplicity is upper semi-continuous on $\Spec R$.
\end{corollary}

We note the following corollary of semi-continuity.

\begin{corollary}
Let $R$ be a Noetherian ring and suppose the Hilbert-Kunz multiplicity is upper semi-continuous on $\Spec R$.
Then the Hilbert-Kunz multiplicity satisfies the ascending chain condition on $\Spec R$, i.e.
any increasing sequence $e_1 = \ehk(\mf p_1) \leq e_2 = \ehk(\mf p_2) \leq \ldots$ stabilizes. 

In particular, the Hilbert-Kunz multiplicity attains its maximum on $\Spec R$.
\end{corollary}
\begin{proof}
Since $\ehk$ is upper semi-continuous 
$U_i = \{\mf p \mid \ehk(\mf p) < e_i\}$ form an increasing sequence of open sets, so it stabilizes.
\end{proof}

\begin{remark}\label{F-signature remark}
In \cite{Tucker}, Kevin Tucker asked if F-signature is lower semi-continuous in F-finite rings. 
One could hope that the ideas of this paper are extendable for F-signature, 
but, at the present moment, we know nothing about the convergence rate of the F-signature of a local ring.

In fact, one could even ask if the splitting numbers  can be written as
\[a_e = r_Fq^{h} + O(q^{h - 1}),\]
where $r_F$ is the $F$-splitting ratio, $h = \alpha(R) + \dim(R/P)$, and $P$ is the splitting prime of $R$, 
see \cite{Tucker} for more details.
\end{remark}

\begin{remark}
We proved Conjecture~\ref{semi} for the F-finite case and algebras over an excellent local ring 
and want to discuss further difficulties.
At the present moment, the author does not see any way to prove the conjecture in full generality, 
for an arbitrary excellent ring. 

The problem stems from the known proof of existence of the Hilbert-Kunz multiplicity,
both the original paper (\cite{Monsky}) and its refinement (\cite{Tucker}) prove
existence of the limit for a local ring by reducing to a faithfully flat F-finite extension obtained by 
extending the residue field. 
Thus, there is not much connection between these objects for different localizations,
so the results and methods of the present paper cannot be applied.

Furthermore, it is not enough to have a global faithfully flat F-finite extension; 
we needed to use the Gamma construction in order to have an extension that has suitable properties.
\end{remark}

\specialsection*{Acknowledgements}
The author would like to thank Craig Huneke for suggesting the problem and continuous support,
Kevin Tucker for suggestions and discussion of the results, and Ian Aberbach for comments.

\bibliographystyle{plain}
\bibliography{semibib}

\end{document}